\RequirePackage{ifpdf}
\ifpdf 
\documentclass[pdftex]{sigma}
\else
\documentclass{sigma}
\fi

\numberwithin{equation}{section}

\begin{document}

\allowdisplaybreaks
	
\renewcommand{\PaperNumber}{127}

\FirstPageHeading

\renewcommand{\thefootnote}{$\star$}

\ShortArticleName{On $1$-Harmonic Functions}

\ArticleName{On $\mathbf{1}$-Harmonic Functions\footnote{This paper is a
contribution to the Proceedings of the 2007 Midwest
Geometry Conference in honor of Thomas~P.\ Branson. The full collection is available at
\href{http://www.emis.de/journals/SIGMA/MGC2007.html}{http://www.emis.de/journals/SIGMA/MGC2007.html}}}

\Author{Shihshu Walter WEI}

\AuthorNameForHeading{S.W. Wei}

\Address{Department of Mathematics, The University of Oklahoma, Norman, Ok 73019-0315, USA}

\Email{\href{mailto:wwei@ou.edu}{wwei@ou.edu}}

\ArticleDates{Received September 18, 2007, in f\/inal form December
17, 2007; Published online December 27, 2007}

\Abstract{Characterizations of entire subsolutions for the
$1$-harmonic equation of a constant $1$-tension f\/ield are given
with applications in geometry via transformation group theory. In
particular, we prove that every level hypersurface of such a
subsolution is calibrated and hence is area-minimizing over
$\mathbb{R}$; and every $7$-dimensional $SO(2)\times
SO(6)$-invariant absolutely area-minimizing integral current in
$\mathbb{R}^8$ is real analytic. The assumption on the $SO(2)
\times SO(6)$-invariance cannot be removed, due to the f\/irst
counter-example in $\mathbb{R}^8$, proved by Bombieri, De
Girogi and Giusti.}

\Keywords{$1$-harmonic function; $1$-tension f\/ield; absolutely area-minimizing integral current}

\Classification{53C40;  53C42}

\section{Introduction}

The study of $1$-harmonic functions, or more generally that of
$p$-harmonic maps is an area of an active research that is related
with many branches of mathematics.  For instance, in a celebrated
paper of Bombieri, De Girogi and Giusti \cite{BDG}, a $1$-harmonic
function has been constructed to provide a counter-example for
interior regularity of the solution to the co-dimension one
Plateau problem in $\mathbb{R}^n$ for $n > 7$. Recall a $C^1$
functions $f: \mathbb{R}^n \to \mathbb{R}\, $ is said to be \emph
{$1$-harmonic} if it is a~weak solution of $1$-harmonic equation
\begin{gather}
\mbox{div}\left(\frac {\nabla f} { |\nabla f|}\right)= 0\,
,\label{6.0}
\end{gather}
where $|\nabla f|$ is the length of the gradient $\nabla f$ of $f$,
and for a $C^2$ function $f$ without a critical point, $ \mbox{div}
\left(\frac {\nabla f} { |\nabla f|}\right)\, $ is said to be the
\emph {$1$-tension field} of $f$.

In this paper, characterizations of entire subsolutions for the
$1$-harmonic equation of a constant $1$-tension f\/ield  are given
in various aspects, and their relationships with calibration
geomet\-ry are established (cf.~Theorem~\ref{T:6.2}, Corollary
\ref{C:6.2}). As applications, we prove via transformation group
theory (cf.~\cite{H,HL,L2,B,W}) that
the cone over $S^1 \times S^5$ is not minimizing in $\mathbb{R}^8$
but is stable; that any $7$-dimensional $SO(2) \times
SO(6)$-invariant absolutely area-minimizing integral current in~$\mathbb{R}^8$ is real analytic; and that the only $7$-dimensional
$SO(3) \times SO(5)$-invariant minimi\-zing integral current with
singularities in $\mathbb{R}^8$ is the cone over $S^2 \times S^4$, and is minimizing over~$\mathbb{R}$ (cf.~Theo\-rems~\ref{T:6.3}--\ref{T:6.4}). These results improved an early partial
proof by numerical computation done by Plinio 
Simoes \cite{S} in his Berkeley thesis. The assumption on the
$SO(2) \times SO(6)$-invariance cannot be removed, due to the
f\/irst counter-example of Bombieri, De Girogi and Giusti that the
cone over $S^3(\frac 1{\sqrt 2}) \times S^3(\frac 1{\sqrt
2})\subset S^7(1)$ is area-minimizing in $\mathbb{R}^8$. It
should be pointed out that 
Fang-Hua Lin~\cite{Li} proved
that the cone over $S^1 \times S^5$ is one-sided area-minimizing
and is stable by a dif\/ferent method. By constructing $1$-harmonic
functions on hyperbolic space $H^n$, $H^n \times H^n$, $H^n
\times SO(n,1)$ and many other associated spaces, S.P.~Wang and
the author~\cite{WaW} show the Bernstein Conjecture in these
spaces to be false in all dimensions. In particular, these
constructions give the~\emph {first} set of examples of complete,
smooth, embedded, minimal \mbox{(hyper-)}sur\-faces in hyperbolic space
$H^n$ in all dimensions (cf.~also Remark~\ref{rem.3.3}(ii)).

\section{Fundamentals in geometric measure theory}

For our subsequent development, we recall some fundamental facts,
def\/initions, and notations, for which the reference is Federer's
book \cite{F2} and paper \cite{F4}.

Let $N$ denote an $n$-dimensional Riemannian manifold and denote
by $\mathcal{R}_p^{\rm loc}(N)$ the set of $p$-dimensional, locally
rectif\/iable currents (of Federer and Fleming, cf.~\cite{FF}) on $N$.
For $S \in \mathcal{R}_p^{\rm loc}(N)$, denote the mass of $S$ by
$\mathbf{M}(S)$, and the ${\emph boundary}$ of $S$ by $\partial
S$, and is given by $(\partial S)( w)=S(d w)$, where
$w$ is a smooth $p$-form and $d$ is the exterior
dif\/ferentiation. From a calculus of variational viewpoint, we make
the following

\begin{definition}
A current $T \in \mathcal{R}_k^{\rm loc}(N)$ is said to be
\emph {stationary} if $\frac
{d}{dt}\mathbf{M}({\phi^V_{t_*}}(T))|_{t=0}$ for all vector f\/ields
$V$ on $N$ with compact support where $\phi^V_t$ is the f\/low
associated with $V$, and \emph {stable} if for every vector f\/ields
$V$ on $N$ with compact support, there exists an $\epsilon
> 0$ such that $\mathbf{M}(T) \leq \mathbf{M}({\phi^V_{t_*}}(T))$
for $|t| < \epsilon$.
\end{definition}

We are primarily interested in minimizing currents.

\begin{definition}
A current $T \in \mathcal{R}_k^{\rm loc}(N)$ is
\emph{homologically (resp.\ absolutely) area-minimizing} over~$\mathbb{Z}$ if for all compact sets $K \subset M$, we have
$\mathbf{M}(\phi_KT) \leq \mathbf{M}((\phi_KT)+S)$ for all $S \in
\mathcal{R}_k^{\rm loc}(N)$ having compact support and being the
boundary of some current in $\mathcal{R}_{k+1}^{\rm loc}(N)$ with
compact support (resp. the empty boundary)(here $\phi_K$ denotes
the characteristic function on $K$).
\end{definition}

Using a dimension reduction technique, Federer proves that the
support of an area-minimizing integral current $T$~\cite{FF} minus
another compact set $\mathrm{S}$ whose Hausdorf\/f dimension does not
exceed $n-8$ is an $(n-1)$-dimensional analytic manifold~\cite{F3}.
Hence, if $n\le 7$, then $\mathrm{S}= \varnothing$. If $n = 8$,
$\mathrm{S}$ consists of at most isolated points~\cite[5.4.16]{F2}. This
result is optimal by the counter-example due to Bombieri--De Giorgi--Giusti \cite{BDG}
 that  $ \{x\in
\mathbb{R}^{2m}:x^2_1+\cdots+x^2_m=x^2_{m+1}+\cdots+x^2_{2m}\}$ is
an area-minimizing cone over the product of $(m-1)$-spheres
$\big\{x\in\mathbb{R}^{2m}:x^2_1+\cdots+x^2_m=x^2_{m+1}+\cdots{}$ ${}+x^2_{2m}=\frac12\big\}$
in $\mathbb{R}^{2m}$ for $m\geq 4$.

\smallskip

The union of the groups $\mathcal{F}_{m, K}(U)=\{R + \partial T: R
\in \mathcal{R}_{m, K}(U), T \in \mathcal{R}_{m+1, K}(U)\}$
corresponding to all compact $K \subset U$ is the group
$\mathcal{F}_m(U)$ of $m$-dimensional \emph {integral flat chains
in an open subset}~$U$ of $\mathbb{R}^n$. We denote the group of
$m$-dimensional \emph {integral flat chains, cycles and
boundaries} by $\mathcal{F}_m(A)=\mathcal{F}_m(\mathbb{R}^n) \cap
\{S: {\rm spt}\, S \subset A\}$, $\mathcal{Z}_m(A, B)=\mathcal{F}_m(A)
\cap \{S:
\partial S \subset \mathcal{F}_m(B) \ \ \hbox {or} \ \ m=0\}$, and $\mathcal{B}_m(A, B)=\{R +
\partial T: R \in \mathcal{F}_m(B), T \in \mathcal{F}_{m+1}(A)\}\,  \hbox
{respectively}.$ Similarly, we def\/ine and denote $\mathbf{F}_m(A)$,
$\mathbf{Z}_m(A, B)$ and $\mathbf{B}_m(A, B)$ the vector space of
$m$-dimensional \emph {real flat chains, cycles and boundaries}
respectively, where $B \subset A$ are compact Lipschitz
neighborhood retract in~$U$.

For every positive convex parametric integrand $\psi$, and every
compact subset $K$ of $A$,  we def\/ine $\mathcal{Z}_{m, K}(A,
B)=\mathcal{Z}_m(A, B) \cap \{R: \hbox {spt} R \subset K\}$,
$\mathcal{B}_{m, K}(A, B)=\mathcal{B}_m(A, B) \cap \{R: \hbox{spt}\, R \subset K\}$, $\mathbf{Z}_{m, K}(A, B)=\mathbf{Z}_m(A,
B) \cap \{R: \hbox{spt}\, R \subset K\},$ and $\mathbf{B}_{m, K}(A,
B)=\mathbf{B}_m(A, B) \cap \{R: \hbox{spt}\, R \subset K\}$, and
make the following

\begin{definition}
An $m$-dimensional rectif\/iable current $Q$ (resp. $Q^{\prime}$) is
said to be \emph {absolutely (resp. homologically) $\psi$-minimizing
in $K$ with respect to $(A, B)$ $over\ \mathbb{Z}$} if
\begin{gather*}
\int_Q \psi =  \inf \left\{ \int _S \psi : S \in \mathcal{F}_{m,
K}(U), Q-S \in \mathcal{Z}_{m, K}(A, B) \right\}\\
\left( \hbox {resp.} \ \ \int_{Q^{\prime}}\psi =  \inf \left\{ \int _S
\psi : S \in \mathcal{B}_{m,K}(U), Q^{\prime}-S \in \mathcal{B}_{m,
K}(A, B)\right\} \ \right).
\end{gather*}
\end{definition}
\begin{definition}
An $m$-dimensional real f\/lat chain $Q$ (resp.~$Q^{\prime}$) is said
to be \emph {absolutely (resp. homologically) $\psi$-minimizing in
$K$ with respect to $(A, B)$ \underbar{over $\mathbb{R}$}} if
\begin{gather*}
\int_Q \psi = \inf \left\{ \int _S \psi : S \in \mathbf{F}_{m,
K}(U), Q-S \in \mathbf{Z}_{m, K}(A, B) \right\}\\
\left( \ \hbox {resp.} \ \ \int_{Q^{\prime}}\psi = \inf\left\{ \int _S \psi
: S \in \mathbf{B}_{m,K}(U), Q^{\prime}-S \in \mathbf{B}_{m, K}(A,
B)\right\} \ \right).
\end{gather*}
\end{definition}
 We will make comparisons between real and integral absolute (resp.
homological) minimizing currents in the subsequent Sections~\ref{sec.3}, \ref{sec.4}, and~\ref{sec.5}.

\section[Characterizations of subsolutions  for $1$-harmonic
equation of constant $1$-tension field]{Characterizations of subsolutions  for $\mathbf{1}$-harmonic
equation\\ of constant $\mathbf{1}$-tension f\/ield} \label{sec.3}

We
connect an entire subsolution  of this sort, with a calibration.
Recall a calibration is a closed form with comass $1$.

\begin{lemma} \label{lem1.1}
Let $M$ be a complete noncompact Riemannian
manifold. For any $x_0 \in M$ and any pair of positive numbers
$s$, $t$ with $s < t$, there exists a rotationally symmetric Lipschitz
continuous function $\psi(x)=\psi(x;s,t)$ and a constant $C_1 > 0$
(independent of $x_0$, $s$, $t$) with the properties:
\begin{gather}
(i)~ \ \   \psi \equiv 1 \ \  on \ \ B(x_0;s),  \ \ and \ \ \psi \equiv 0 \ \ \mbox{off}
\ \ B(x_0;t);\nonumber\\
(ii) \ \  \label{3.0}| \nabla \psi|
\leq \frac{C_1}{t-s}, \ \ a.e.  \ \ on \ \ M.
\end{gather}
\end{lemma}

\begin{proof}
(cf. Andreotti and Vesentini \cite{AV}, Yau \cite{Y}, Karp \cite{K}).
\end{proof}

\begin{theorem}
Let $\Omega$ be a domain in $\mathbb{R}^n$ containing a ball
$B(x_0,r)$ of radius $r$, centered at $x_0$, and $g: \Omega
\to \mathbb{R}$ be a continuous function with $g \ge 0$, and $c
= \inf\limits_{x\in B(x_0,\frac r2)} g(x)$. Let $f: \Omega \to
\mathbb{R}$ be a~$C^1$ weak solution of
\begin{align}\label{3.1}
\mbox{\rm div}\left(\frac {\nabla f} { |\nabla f|}\right)= g(x)  \qquad
\mbox{on} \ \ \Omega,
\end{align}
then the inf\/imum $c$ satisf\/ies
\begin{align*}
0 \le c \le \frac{C_12^n}r,
\end{align*}
where $C_1$ is as in \eqref{3.0}.
\end{theorem}
\begin{proof}  Let $\psi \ge 0$ be as in Lemma \ref{lem1.1}, in which
$M = \mathbb{R}^n$, $t=r$, $s= \frac r2$. Choose $\psi$ to be a test function  in
the distribution sense of \eqref{3.1}. Then via the assumption on $g$,
and Cauchy--Schwarz inequality we have:
\begin{gather*}  \int _{B(x_0,\frac r2)}
c \psi(x) dx  \le \int _{B(x_0,\frac r2)} g(x)\psi(x) dx \\
\phantom{\int _{B(x_0,\frac r2)}
c \psi(x) dx}{} \le \int
_{B(x_0,r)} g(x)\psi(x) dx     = - \int _{B(x_0,r)}
\frac {\nabla f} { |\nabla f|}\cdot \nabla \psi dx \le  \int
_{B(x_0,r)} |\nabla \psi| dx .\end{gather*}
Hence,
\[ c \mbox{Vol}\left(B\left(x_0,\frac r2\right)\right) \le \frac{C_1}{r}\mbox{Vol}
(B(x_0,r))\label{11.9}
\] yields the desired.
\end{proof}
\begin{corollary}\label{C:6.1}
Let $f: \mathbb{R}^n \to \mathbb{R}$ be a $C^1$ weak subsolution of
$1$-harmonic equation \eqref{6.0} with constant $1$-tension field $c$, i.e.\  $ 0 \le \mbox{\rm div} \left(\frac {\nabla f} { |\nabla
f|}\right)= c$ in the distribution sense. Then $f$ is a~$1$-har\-monic function.
\end{corollary}

\begin{corollary}\label{C:6.0}
There does not exist a $C^1$ weak subsolution  $f: \mathbb{R}^n
\to \mathbb{R}$ of equation \eqref{3.1} with $\lim\limits _{r\to \infty}
\inf\limits_{x \in B(x_0,r)} g(x) > 0$, for any $x_0 \in
\mathbb{R}^n$.
\end{corollary}

Let $A \subset R^n$ be an open set.  We denote $BV_{\rm loc}(A)=\{f \in
L_{\rm loc}^1(A)$: the distributional derivatives $D_if$ of $f$ are
(locally) measures\}= $\{f \in L_{\rm loc}^1(A): \hbox {supp}\, \phi_n
\subset K \subset  A$, $\phi_n \rightarrow 0$ uniformly, imply
$\left(\frac {\partial}{\partial x_i}f\right) \phi_n \rightarrow
0\}$. Let $Df=(D_1f, \dots, D_nf)$ denote the gradient of $f$ in
the sense of distributions and $|Df|$ the scalar measure def\/ined by
$\int_K |Df|=\hbox {sup} \int_K \sum_i \epsilon_i(x)D_i f$, where
the supremum is taken over all sets $\{\epsilon_i(x), \; i=1, \dots,
n\}$ of $C^{\infty}(K)$ functions which satisfy $\sum
\epsilon_i^2(x) \leq 1$.
\begin{definition} A function $f \in
BV_{\rm loc}(A)$ {\it has least gradient} in $A$ if for every $g \in
BV_{\rm loc}(A), $ with compact support $K \subset A$ we have
\begin{equation}\int_K |Df| \leq \int_K |D(f+g)|.\label{6.1}\end{equation}\end{definition}

\begin{definition} Let $E$ be a set in
$\mathbb{R}^n$ and $\phi_E$ its characteristic function.  $E$ {\it
has an oriented boundary of least area} with respect to $A$, if
${(i)}\, $ $\phi_E \in BV_{\rm loc}(A)$ and ${(ii)}$ for each $ g
\in BV_{\rm loc}(A)$  with compact support $K \subset A$ we have
$\int_K |D\phi_E| \leq \int_K |D(\phi_E +g)|$.
\end{definition}

\begin{theorem}\label{T:6.2}
Let $f \in H_{\rm loc}^{1, 1} (\mathbb{R}^n)$, and $\nabla f(x) \neq
0$ for every $x$ in $\mathbb{R}^n $. Let $E_{\lambda} = \{x:
f(x) \geq \lambda\}$, and $S_{\lambda}=\{x: f(x) = \lambda\}$.
We denote the set of integers by $\mathbb{Z}$. Then the following
thirteen statements \eqref{1}--\eqref{10} are equivalent and each of them
implies the fourteenth statement~\eqref{14}.
 \renewcommand{\labelenumi}{{\rm \theenumi.}}
\begin{enumerate}\itemsep=0pt
\item \label{1} $f: \mathbb{R}^n \to \mathbb{R}$ is a $C^1$ weak
subsolution of \eqref{6.0} with constant $1$-tension field.
\item \label{2} $f$ is a $C^1$ weak solution of \eqref{6.0} on $\mathbb{R}^n$.
\item  \label{3} $f$ is a $C^1$ $1$-harmonic function on $\mathbb{R}^n$.
\item  \label{4} For each $(a,t_0)=(a_1, \dots, a_{n-1}, t_0)\in
S_{\lambda} $, there exists a neighborhood $\mathcal{D}$ of $a$ in
$\mathbb{R}^{n-1}$, and a~unique real analytic function $\eta :
\mathcal{D}\to \mathbb{R}$ such that $\eta(a) = t_{0}$, $f(x_1,
\dots, x_{n-1}, \eta(x_1, \dots,$ $ x_{n-1})) = \lambda$ and
$\mbox{\rm div} \left(\frac {\nabla \eta}{\sqrt { 1+|\nabla \eta|^2}}
\right) = 0\, $ on $\mathcal{D}$.
\item \label{11} Each level hypersurface
$S_{\lambda}$ is minimal in $\mathbb{R}^n$.
\item \label{5} $\frac
{*df}{|df|}$ is a globally defined ``weakly'' closed form with
comass $1$.
\item \label{12} $f$ is a function of least gradient in $\mathbb{R}^n$.
\item \label{13} Each $E_{\lambda}$, $\lambda \in \mathbb{R}$ has an
oriented boundary of least area with respect to $\mathbb{R}^n$.
\item \label{8} Each level hypersurface $S_{\lambda}$ is
absolutely area-minimizing in $\mathbb{R}^n$ over $\mathbb{Z}$.
\item \label{6} Each level hypersurface $S_{\lambda}$ is
absolutely area-minimizing in $\mathbb{R}^n$ over $\mathbb{R}$.
\item \label{7} Each level hypersurface $S_{\lambda}$ is
homologically area-minimizing in $\mathbb{R}^n$ over $\mathbb{R}$.
\item \label{9}
 Each level hypersurface $S_{\lambda}$ is homologically
area-minimizing in $\mathbb{R}^n$ over $\mathbb{Z}$.
\item \label{10} Each level hypersurface $S_{\lambda}$ is stable in
$\mathbb{R}^n$.
\item \label{14} If $f \in C^2(\mathbb{R}^n)$,
then $\frac {*df}{|df|}$ is closed and the restriction $\left.\frac
{*df}{|df|}\right|_{S_{\lambda}}$ is its volume form, hence each
${S_{\lambda}}$ is real absolutely area-minimizing in $\mathbb{R}^n$
over $\mathbb{R}$.
\end{enumerate}
\end{theorem}
\begin{corollary}\label{C:6.2}
Every level hypersurface of a $C^2$ subsolution  of $1$-harmonic
equation on $\mathbb{R}^{n+1}$ with constant $1$-tension field is
calibrated and hence is area-minimizing over $\mathbb{R}$.
\end{corollary}
\begin{proof}  $\eqref{1} \Leftrightarrow \eqref{2} \Leftrightarrow \eqref{3}:$ This
follows immediately from Corollary \ref{C:6.1}.

$\eqref{2} \Leftrightarrow \eqref{4}:\, $ $(\Rightarrow)$ Let
$f(x_1, \dots, x_{n-1}, t)=\eta(x_1, \dots, x_{n-1})-t$. The
assertion follows from the implicit function theorem and
\begin{equation} \label{6.2}
0 =
\int \frac {\sum\limits^{n-1}_{i=1}\frac{\partial f}{\partial
x_i}\frac {\partial \varphi}{\partial x_i}}{|\nabla f|} + \int
\frac {\frac {\partial f}{\partial t}}{|\nabla f|} \frac {\partial
\varphi}{\partial t} = \int \sum\limits^{n-1}_{i=1} \frac {\frac
{\partial \eta}{\partial x_i}}{\sqrt {1+|\nabla \eta|^2}} \frac
{\partial \varphi}{\partial x_i}
\end{equation}  for all $\varphi
\in C^{\infty}_0 (\mathcal{D}\times \mathbb{R})$.  The regularity
of solutions of minimal surface equation implies that $\eta$ is
real analytic and completes the proof. $(\Leftarrow)$ This follows
immediately from \eqref{6.1}.

 $\eqref{4} \Leftrightarrow \eqref{11}:\, $ This is due to the
fact that the graph of a solution to the minimal surface equation
on $\mathcal{D}$ is a minimal hypersurface in $\mathcal{D} \times
\mathbb{R}$.

 $\eqref{2}\Leftrightarrow\eqref{5}:$ This follows from the
following: For every $\phi \in C_0^{\infty}(A)$,
\begin{gather*}
 \int\limits_A \frac{\ast df}{|df|}\wedge d\phi
=\int\limits_A \sum\limits^n_{i,j=1}(-1)^{i-1} \frac {\frac
{\partial f}{\partial x_i}}{|\nabla f|} dx^1\wedge \dots \wedge
dx^{i-1} \wedge dx^{i+1} \wedge \dots \wedge dx^n \wedge \frac
{\partial\phi}{\partial x_j} dx^j\\
\phantom{\int\limits_A \frac{\ast df}{|df|}\wedge d\phi}{}
=\int\limits_A \sum\limits^n_{i=1}(-1)^{n-1} \frac {\frac {\partial
f}{\partial x_i} \frac {\partial \phi}{\partial x_i}}{|\nabla f|}
dx^1 \wedge \dots \wedge dx^i \wedge \dots \wedge dx^n.\label{6.3}
\end{gather*}

$\eqref{2}\Rightarrow\eqref{12}$: let us f\/irst assume that $g \in C^1_0(A)$.
Let $h(t)= \int |D(f+tg)|$. Then
\begin{gather*}
h^{\prime}(t)=\int \frac{\left(\sum\limits_{i=1}^n \frac
{\partial(f+tg)}{\partial x_i} \frac {\partial g}{\partial x_i}\right)}{
\left(\sum\limits_{i=1}^n \left(\frac {\partial (f+tg)}{\partial x_i}\right)^2 \right)^\frac
12} .
\end{gather*}
Hence $h^{\prime}(0)=0$ by assumption. Furthermore,
\begin{gather*}
h^{\prime\prime}(t)=\int \frac {\left(\sum\limits^n_{i=1}(\frac
{\partial g}{\partial x_i})^2\right)\left(\sum\limits^n_{i=1}(\frac
{\partial (f+tg)}{\partial x_i})^2\right) -
\left(\sum\limits_{i=1}^n \frac {\partial (f+tg)}{\partial x_i}
\frac {\partial g}{\partial
x_i}\right)^2}{\left[\sum\limits_{i=1}^n(\frac{\partial(f+tg)}{\partial
x_i})^2\right] ^{\frac 32}} \geq 0 ,
\end{gather*}
by the Cauchy--Schwarz inequality. Therefore $\int |Df|=h(0) \leq
h(1) = \int |D(f+g)|.$ If $g \in BV_{\rm loc}(A)$ with compact support
$K$ and let $Dg=G_1+G_2$ where $G_1$ is completely continuous and
$G_2$ is the singular part of $Dg$ with support $N_g$ of measure
zero. Then we have $\int_K |D(f+g)|= \int_K |Df+G_1| + \int_K
|G_2|  $ because $f \in H^{1, 1}_{\rm loc}(A)$. Let $g_{\varepsilon} =
g*\psi_{\varepsilon}$ where $\psi_{\varepsilon}$ is a mollif\/ier.
Then $g \in C_0^1(A)$ and $\int_{K_{\epsilon}}|Df| \leq
\int_{K_\epsilon}|D(f+g_{\epsilon})| \leq \int_{K_{\epsilon}}|Df+G_1
\ast \Psi_{\epsilon}| + \int_A|G_2 \ast \Psi_{\epsilon}| $, where
$K_{\epsilon}=\{x \in A: \hbox {dist} (x, K) < \epsilon\}$. Letting
$\epsilon \rightarrow 0$ completes the proof (cf.~\cite{BDG}).

$\eqref{12}\Rightarrow\eqref{13}:$  This follows from Coarea
formula for BV functions \cite{M1}, $\int_K\!
|Df|=\int\limits^{\infty}_{-\infty}\! (\int_K
|D\phi_{\lambda}|)d\lambda\! $ together with two observations:

(i) If
$f_1$ and $f_2$ satisfy \eqref{6.1}, so does $\sup(f_1, f_2)$.

(ii) If $f_i \in BV_{\rm loc}(A), f_i \rightarrow f$ in
$L^1_{\rm loc}$ and each $f_i$ satisf\/ies \eqref{6.1}, so does also $f \in
BV_{\rm loc}(A)$ and satisf\/ies \eqref{6.1}.

For detailed proof see \cite{M2}.

  $\eqref{13}
\Rightarrow \eqref{8}:$ Let $\phi_{\lambda}= \phi_{E_{\lambda}}$.
Since for every $x$ in $\mathbb{R}^n$, $\nabla f(x) \neq 0$,
$\partial E_{\lambda}=S_{\lambda}$ for $S_{\lambda} \neq \varnothing$.  It
follows from a theorem of Miranda \cite{M1} that on any compact set $K$
in $\mathbb{R}^n  $,  the Hausdorf\/f $(n-1)$-measure
\[
\mathcal{H}^{n-1}(K \cap S_{\lambda})=\int_K |D\phi_{\lambda}| \leq \int_K |D(\phi_{\lambda} +
g)|=\mathcal{H}^{n-1}(K \cap T)
\]
for all sets $T$ with $\partial(K \cap T)= \partial (K \cap
S_{\lambda})$.

  $\eqref{8}\Rightarrow\eqref{6}:$  It follows from Theorem
\ref{T:6.6}.

  $\eqref{6}\Rightarrow\eqref{7}\Rightarrow\eqref{9}:$ Since absolute
area-minimization over $\mathbb{R}\Rightarrow$ homological
area-mini\-mi\-za\-tion over $\mathbb{R}\Rightarrow$ homological
area-mini\-mi\-za\-tion over $\mathbb{Z}$.

  $\eqref{9}\Rightarrow\eqref{10}\Rightarrow\eqref{11}:$  Since
homological minimization over $\mathbb{Z}\Rightarrow$ stability
$\Rightarrow$ minimality. This completes the proof of $\eqref{1}
\Leftrightarrow \cdots \Leftrightarrow \eqref{10}$.

  $\eqref{2}\Rightarrow\eqref{14}:$ If $f \in C^2(A)$ then by
\eqref{6.2} $\frac{\ast df}{|df|}$ is closed. Now let $e_1, \dots,
e_{n-1}$ be an orthonormal basis for the tangent space of
$S_\lambda$ at $x_0$ and $\nu$ a unit normal vector at $x_0$.  We
denote by tilde ``$\sim$'' the canonical isomorphism between a
tangent space and its dual space.  To show $\frac{\ast df}{|df|}$
has comass $1  $, note for any $(n-1)$-vector f\/ield $\xi$,
\begin{gather*}
\frac{\ast df}{|df|}(\xi)=\left(\ast \frac {\widetilde{\nabla
f}}{|\nabla f|}\right) (\xi) \hspace{5mm}\left(\hbox{because } \frac
{df}{|df|}(X)=\frac {Xf}{|\nabla f|}=\left\langle \frac
{\nabla f}{|\nabla f|}, X \right\rangle\right)\\
\phantom{\frac{\ast df}{|df|}(\xi)}{} =(\ast \tilde{\nu}) (\xi) = (\widetilde{e_1 \wedge \dots \wedge
e_{n-1}})(\xi) = \langle e_1 \wedge \dots \wedge e_{n-1}, \xi\rangle.
\end{gather*}

In particular $\frac{\ast df}{|df|}(e_1 \wedge \dots \wedge
e_{n-1})=1 $, $\frac{\ast df}{|df|} (\xi) \leq 1$
 and $\left.\frac{\ast df}{|df|}\right|_{
S_{\lambda}}= \hbox{\text volume\,  element\, of\, }
S_{\lambda}$. By the formalism of Stokes theorem, for any
integral current $T$ with $\partial T = \partial (S_{\lambda} \cap
B_r)$
\begin{gather*}
M(S_{\lambda} \cap B_r) =(S_{\lambda} \cap B_r)\left(\frac{\ast
df}{|df|}\right) =
T\left(\frac{\ast df}{|df|}\right)\\
\phantom{M(S_{\lambda} \cap B_r)}{} =\int \frac{\ast
df}{|df|}\left(\overrightarrow{T_x}\right)d||T||(x) \leq \int
d||T||=M(T),
\end{gather*}
 where $\overrightarrow{T}$ is the
f\/ield of oriented unit tangent planes to $T$.
\end{proof}

\begin{remark}  In Theorem \ref{T:6.2}, if one replace $\mathbb{R}^n$ with an open subset $A$ in $\mathbb{R}^n$, then assertions \eqref{2}$\Leftrightarrow \cdots \Leftrightarrow
\eqref{10} \Rightarrow \eqref{14}$ remain to be true.
\end{remark}

\begin{remark}\label{R:6.2}  Concerning the assertion $\eqref{2}
\Rightarrow \eqref{12}$, a stronger theorem can be found in \cite{BDG}: Let
$A \subset R^n$ be an open set and let $f \in H^{1,1}_{\rm loc}(A)$.
Suppose that $(i)$~$\mathcal{H}_n(\{x \in A: |\nabla
f|=0\})=0 $, $(ii)$~$\mathcal{H}_{n-1}(N)=0$ where $N$ is a
closet set in $A  $, $(iii)$~$\int_{A-N} |\nabla f|^{-1}
\sum\limits ^n_{i=1} \frac {\partial f}{\partial x_i} \frac
{\partial \phi}{\partial x_i} dx=0$ for every $\phi \in C^1_0(A-N)$.
Then $f$ has least gradient with respect to $A$.
\end{remark}

\begin{remark}\label{rem.3.3}  (i)  The assertion $\eqref{12}
\Rightarrow \eqref{8}$ is due to Miranda.

(ii)  Connecting the
assertions $\eqref{11}$,  $\eqref{5}$, and $ \eqref{9}$ on Riemannian
manifolds, S.P.~Wang and the author \cite{WaW} prove that  if each level
hypersurface of a smooth function $f: M \to \mathbb{R}$ on an
oriented $\emph{Riemannian manifold}$ $M$ with nowhere
vanishing $\nabla f   $, is minimal, then there exists a closed
form with comass $1$ on $M$ and hence each level hypersurface is
homologically area-minimizing over $\mathbb{R}$.
\end{remark}

\begin{corollary}
  Let $A$ be an open subset in $\mathbb{R}^n$, $N$ be
a closed subset in $A$ with $\mathcal{H}_{n-1}(N)=0$.  Then the
graph of any weak solution of the minimal surface equation
$\sum\limits ^n_{i=1} \frac {\partial}{\partial x_i} \left(\frac
{\frac{\partial f} {\partial x_i}}{\sqrt {1+|\nabla f|^2}}\right)=0$
on $A-N$ is in fact absolutely area-minimizing in $A \times
\mathbb{R} \subset \mathbb{R}^{n+1}$ over $\mathbb{R}$.
\end{corollary}
\begin{proof}
Applying \eqref{6.2} in which ``$f(x_1, \dots, x_{n-1}, t)=\eta(x_1,
\dots, x_{n-1})-t$'' is replaced with \linebreak ``$F(x_1, \dots, x_n,
t)=f(x_1, \dots, x_n)-t$'', and Remark \ref{R:6.2}, we have that
$F$ is a $C^1$ $1$-harmonic function in $A$. By Theorem
\ref{T:6.2}, the zero level set $S_0=\{(x_1, \dots, x_n, t) : t =
f(x_1, \dots, x_n)\}$ is absolutely area-minimizing in $A \times
\mathbb{R} \subset \mathbb{R}^{n+1}$ over $\mathbb{R}$.\end{proof}

\section{Further applications}
\label{sec.4}  A natural question arises: Are
Bombieri--De~Giorgi--Giusti and Lawson cones the only  $SO(m)$ $\times\,
SO(n)$-invariant singular absolutely area-minimizing integral
currents in Euclidean \linebreak spa\-ce $\mathbb{R}^{m+n+2}$?  The answer is
af\/f\/irmative. Combining the theory of $1$-harmonic functions
developed, and the techniques of transformation groups in \cite{HL,
L2,B}, and \cite{W}, evolved from the ideas in \cite{H}, one obtains the
following:
\begin{theorem}\label{T:6.3}  The cone $C(S^m \times S^n)$
over $S^m \times S^n$ is the unique singular absolutely
area-minimizing hypersurface in the class of $SO(m+1) \times SO
(n+1)$-invariant integral currents in $\mathbb{R}^{m+n+2}$ over
$\mathbb{R}$ for $m+n > 7$ or $m+n=6$, $|m-n| \leq 2$.  (It is known
that the cone is not even stable otherwise.)
\end{theorem}

\begin{proof}
Assume $m=n$.  Let Lie group $G=SO(n+1)\times SO(n+1)$ acting on manifold
$\mathbb{R}^{n+1}\times \mathbb{R}^{n+1}$ in the standard
way, i.e.\ assigning $\big( (A,B),(x,y) \big)\in G \times
\mathbb{R}^{2n+2}$ to $ (A\cdot x,B\cdot y) \in
\mathbb{R}^{2n+2}  $, where ``$\cdot $'' is the matrix
multiplication. Then the collection $X$ of principle orbits is
given by $X =\{(x,y)\in \mathbb{R}^{2n+2}: |x||y| \ne 0\}  $,
where ``$| \cdot |$'' is the length of ``$\cdot$'' in
$\mathbb{R}^{n+1}  $. The orbit space which is stratif\/ied, can be
represented as $\mathbb{R}^{2n+2}/G = \{(u,v)\in \mathbb{R}^{2}:
u, v \ge 0\}\, = X \cup \{(u,v)\in \mathbb{R}^{2}: u=0, v > 0\}
\cup \{(u,v)\in \mathbb{R}^{2}: u > 0, v = 0\}\,\cup \{(0,0)\}
$. The canonical metric on $\mathbb{R}^{2n+2}/G\, $(compatible
with the f\/ibration over each stratum) is the usual f\/lat one
$ds_0^2 = du^2 + dv^2  $. The canonical projection $\pi :
\mathbb{R}^{2n+2} \to \mathbb{R}^{2n+2}/G$ is given by $\pi (x,y)
= (|x|,|y|)  $, and let $X/G = \pi (X)  $. Then the length of a
curve $\sigma$ in
 $(X/G  , ds_0^2)$ is the length of any orthogonal trajectory through the corresponding orbits in
 $X$,
 and
$2n$-dimensional volume of $\pi^{-1}((u,v))$ (which is
dif\/feomorphic to $S^n \times S^n$) is proportional to $u^n
v^n $, for $(u,v)\in X/G $. Thus if we choose the metric $ds^2
= u^{2n} v^{2n} (du^2 + dv^2)$ on $\mathbb{R}^{2n+2}/G\, ,$ then
by Fubini's theorem, the length of a curve $\sigma$ in
 $(\mathbb{R}^{2n+2}/G  , ds^2)$ is equal to $(2n+1)$-dimensional volume
 of hypersurface $\pi^{-1}\sigma$ (with possible singularities) in
$\mathbb{R}^{2n+2} $, up to a constant factor. It follows that $\sigma$ is a length minimizing geodesic ``downstairs" (in $(\mathbb{R}^{2n+2}/G ,
 ds^2) $), if and only if $\pi^{-1}\sigma$ is area-minimizing \emph{in the class
 of $G$-invariant} $(2n+1)$-dimensional currents ``upstairs''  (in $(\mathbb{R}^{2n+2}, dx_1^2 + \cdots + dx_{2n+2}^2)$), or equivalently, $\pi^{-1}\sigma$ is area-minimizing in
 $(\mathbb{R}^{2n+2} ,dx_1^2 + \cdots + dx_{2n+2}^2)
 $ \emph{in general} (cf.~\cite{L2}, \cite[p.~174, 6.4]{B} and \cite{W}). Furthermore, if a length minimizing geodesic $\sigma$ meets the boundary $\{(u,v)\in \mathbb{R}^{2}:
u=0, v > 0\}  \cup \{(u,v)\in \mathbb{R}^{2}: u > 0, v = 0\}  $,
it meets the boundary orthogonally by the f\/irst variational
formula for the arc-length functional, and the corresponding
$\pi^{-1}\sigma$ is a~regular, embedded and analytic hypersurface
in $\mathbb{R}^{2n+2} $. If $\sigma$ meets the vertex $\{(0,0)\} $, then
$\pi^{-1}\sigma$ is singular. Therefore, it suf\/f\/ices to show that
any curve in $\mathbb{R}^{2n+2}/G $, other than the diagonal ray
emanating from the origin is not absolutely length minimizing with
respect to the metric $ds^2=u^{2n}v^{2n}(du^2+dv^2)$.

\begin{figure}[h]
\centering
\includegraphics{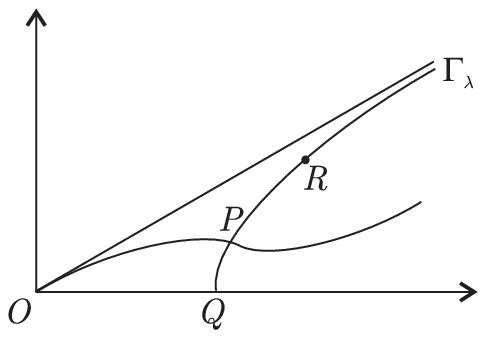}
\end{figure}

Now let $\Gamma = \{( u_0(t), v_0(t))\}$ be the geodesic through
$(1,0)$ in $(\mathbb{R}^{2n+2}/G , ds^2) $, and
$\Gamma_{\lambda}=\{(\lambda u_0(t), \lambda v_0(t))\}$, $\lambda
> 0  $. In \cite{BDG}, a $1$-harmonic function was constructed in such a way
that the lift of family $\{\Gamma_{\lambda}\}$ of these homothetic
geodesics are level hypersurfaces  in $(\mathbb{R}^{2n+2}
,dx_1^2 + \cdots + dx_{2n+2}^2)  $. Hence $\Gamma_{\lambda}$ is
absolutely length minimizing in $(\mathbb{R}^{2n+2}/G ,
 ds^2)$ (cf.~also Theorem~\ref{T:6.2}, Remark~\ref{R:6.2}).
Now suppose Theorem~\ref{T:6.3} were not true.  Then there
would exist a curve $QP \subset \Gamma_{\lambda}$ transverse to a
length minimizing curve $OP$. It follows that the length $l(OP)$ of
$OP$ would satisfy $l(OP)=l(QP)$. Consider the curve $OPR$ where $R$
is on the curve $\Gamma_{\lambda}$, and $l(OPR)=l(QPR) $. Then the
curve $OPR$ would be a geodesic, and hence smooth at $P $. This is
a contradiction. Similarly, one can show the remaining case $m \neq
n$.
\end{proof}

\begin{theorem}
\label{T:6.5} The cone $C(S^1 \times S^5)$ over $S^1
\times S^5$ is not absolutely area-minimizing, although it is
stable.
\end{theorem}

\begin{proof}  Suppose, on the contrary, that the cone were absolutely area-minimizing.
Then consider Lie group $G=SO(2)\times SO(6)$ acting on
manifold $\mathbb{R}^{2}\times \mathbb{R}^{6}$ in the standard
way. By the previous argument, this would imply the line segment
$\overline {OP}$ were length-minimizing in $(\mathbb{R}^{8}/G,
ds^2)$, where $ds^2 = u^2v^6 (du^2 + dv^2) $. On the other
hand, based on the study of Simoes' thesis \cite{S}, \cite{L2} and~\cite{W}, the
level curve $(u_{\lambda}, v_{\lambda})$ in the $u, v$-plane is
absolutely length-minimizing. Argue as before, the curve $OPR$
would be smooth at~$P$. This is a contradiction. The stability of
the cone follows from Simons' work~\cite{Si}.
\end{proof}
\begin{figure}[h]
\centering
\includegraphics{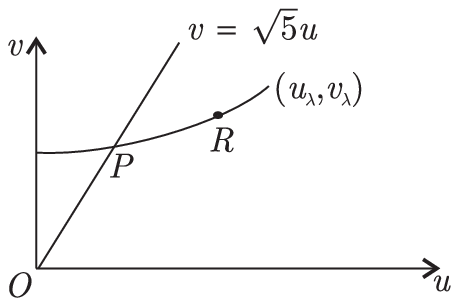}
\end{figure}
\begin{theorem}
\label{T:6.4}  Any $7$-dimensional $SO(2) \times
SO(6)$-invariant absolutely area- minimizing integral current in
$\mathbb{R}^8$ is real analytic.
\end{theorem}
\begin{proof}
By the argument given in the proof of Theorem~\ref{T:6.3},
it suf\/f\/ices to show that any curve in $\mathbb{R}^{2n+2}/G$,
from the origin is not absolutely
length minimizing with respect to the metric
$ds^2=u^{2}v^{6}(du^2+dv^2)$. By Theorem \ref{T:6.5}, the diagonal
ray emanating from the origin is not length minimizing. Similarly,
if there were an absolutely length minimizing curve starting from
the origin lying above $v=\sqrt 5 u$, then this would lead to
an irregularity of a geodesic, a contradic\-tion.
\end{proof}

\section{Comparison theorem}
\label{sec.5} It is known that each level hypersurface of a function
of least gradient def\/ined on an open subset $A \subset \mathbb{R}^n$
is absolutely area-minimizing in $A$ over $\mathbb{Z}$.  It is
tempting to ask it if is absolutely area-minimizing in $A$ over
$\mathbb{R}$. This motivates our discussion on comparison between
real and integral absolute (or homological) minima.  In general they
are distinct. Examples are given by Almgren \cite[5.11]{F4}, Federer \cite{F4}
and Lawson \cite{L1}. Furthermore, in the case of 1-dimensional (or
co-dimension 1) integral f\/lat chains, Federer \cite{F4} has shown that
real and integral homological (or absolute) minimizing are the same.

Let $\overline {M}$ be a locally Lipschitz neighborhood retract in
$\mathbb{R}^n$ (i.e.\ there exists a locally Lipschitz map
which retracts a neighborhood of $\overline {M}$ onto $\overline
{M})$, $M$ be an open subset of $\overline {M}$, and $A$ be
an open subset of $\mathbb{R}^n$. Using the assumption on
vanishing topology, an exhaustion of $M$ by an increasing sequence
of compact set $K_i \subset M$, we obtain the following:
\begin{theorem}\label{T:6.6}  $(1)$ Let $T^{n-1}$ denote a codimension $1$
integral absolutely area-minimizing rectifiable current in $M$
with homology group $H_{n-1}(\overline {M})=0$. Then $T^{n-1}$
is absolutely area-minimizing in~$M$ if and only if $ T^{n-1}$
is absolutely area-minimizing in $A$; and if and only if $T^{n-1}$
is real absolutely area-minimizing in $A$. $(2)$ Let
$H_1(\overline {M})=0$. $T^1$ is a homologically area-minimizing
rectifiable current of degree $1$ of~$M$ if and only if $T^1$ is
real homologically area-minimizing in~$M$.
\end{theorem}
We have the following immediate
\begin{corollary}  The level hypersurface of a function of least
gradient in an open subset $A$ of $\mathbb{R}^n$ is absolutely
area-minimizing over $\mathbb{R}$. \end{corollary}

\begin{corollary}  Let $N$ be a closed set in $A \subset \mathbb{R}^n$ with
$H_{n- 1}(N)=0$.  The graph of any weak solution of the minimal
surface equation $\sum\limits ^n_{i=1} \frac {\partial}{\partial
x_i} \left(\frac {\frac{\partial f}{\partial x_i}}{\sqrt {1+ |\nabla
f|^2}}\right)=0$ on $A-N$ is in fact absolutely area-minimizing in
$A \times \mathbb{R} \subset \mathbb{R}^{N+1}$ {\underbar {over
$\mathbb{R}$}}.
\end{corollary}
\begin{corollary}  All the examples we find in {\rm \cite{W}} are absolutely
area-minimizing over $\mathbb{R}$.
\end{corollary}



\subsection*{Acknowledgements}
Research was partially supported by NSF Award No DMS-0508661, OU
Presidential International Travel Fellowship, and OU Faculty
Enrichment Grant.  The author wishes to thank Professors H.
Blaine Lawson Jr., and Wu-Yi Hsiang for their interest, and
Professor Herbert Federer for his help in Theorem \ref{T:6.6}
which essentially derives from him. The author also wishes to
express his gratitude to the referees and the editors for their
comments and suggestions without which the present form of this
article would not be possible.

\pdfbookmark[1]{References}{ref}

\LastPageEnding

\begin{thebibliography}{99}

\footnotesize\itemsep=0pt

\bibitem{AV} Andreotti A., Vesentini E.,   Carleman estimate for the
Laplace--Beltrami equation on complex manifolds, {\it Inst. Hautes
\'Etudes Sci. Publ. Math.} {\bf 25} (1965), 81--130.

\bibitem{B} Brothers J.E.,  Invariance of solutions to invariant parametric
variational problems, {\it Trans. Amer. Math. Soc.} {\bf 262} (1980), 159--179.

\bibitem{BDG} Bombieri E., de
Giorgi E.,   Giusti E., Minimal cones and the Bernstein
problem, {\it Invent. Math.} {\bf 7} (1969), 243--268.

\bibitem{F1} Federer H.,   Some theorems on
integral currents, {\it Trans. Amer. Math. Soc.} {\bf 117} (1965) 43--67.

\bibitem{F2} Federer H.,    Geometric measure theory, Springer,
Berlin~-- Heidelberg~-- New York, 1969.

\bibitem{F3} Federer H.,   The singular sets of area
minimizing rectif\/iable currents with codimension one and of area
minimizing f\/lat chains modulo two with arbitrary codimension,
{\it Bull. Amer. Math. Soc.} {\bf 76} (1970), 767--771.

\bibitem{F4} Federer H.,   Real f\/lat chains, cochains and
variational problems, {\it Indiana Univ. Math. J.} {\bf 24} (1974), 351--406.

\bibitem{FF} Federer H.,  Fleming W.H.,    Normal and
integral currents, {\it Ann. of Math.~(2)} {\bf 72} (1960), 458--520.

\bibitem{H} Hsiang W.Y.,  On the compact homogeneous minimal
submanifolds, {\it Proc. Natl. Acad. Sci. USA} {\bf 56} (1966), 5--6.

\bibitem{HL} Hsiang W.Y., Lawson H.B. Jr.,  Minimal
submanifolds of low cohomogeneity, {\it J. Differential Geom.} {\bf 5} (1971),
1--38.

\bibitem{K} Karp L.,  Subharmonic functions on real and
complex manifolds, {\it Math. Z.} {\bf 179} (1982),  535--554.

\bibitem{L1} Lawson H.B.,  The stable homology of a f\/lat
torus, {\it Math. Scand.} {\bf 36} (1975), 49--73.

\bibitem{L2} Lawson H.B.,   The equivariant Plateau problem
and interior regularity, {\it Trans. Amer. Math. Soc.} {\bf 173} (1972), 231--249.

\bibitem{Li} Lin F.-H.,  Minimality and stability of
minimal hypersurfaces in $R\sp N$, {\it Bull. Austral. Math. Soc.} {\bf 36}
(1987), 209--214.

\bibitem{M1} Miranda M.,  Sul minimo dell'integrale del
gradiente di una funzione,  {\it Ann. Scuola Norm. Sup. Pisa~(3)} {\bf 19} (1965),
626--665.

\bibitem{M2}  Miranda M., Comportamento delle
successioni convergenti di frontiere minimali, {\it Rend. Sem. Mat. Univ. Padova} {\bf 39} (1967), 238--257.

\bibitem{S} Simoes P.,  A class minimal cones in $\mathbb{R}^n, n \ge 8$, that minimizes area, Thesis,  Berkeley, 1973.

\bibitem{Si} Simons J.,  Minimal varieties in
riemannian manifolds, {\it Ann. of Math. (2)} {\bf 88} (1968), 62--105.

\bibitem{WaW} Wang S.P.,   Wei S.W., Bernstein conjecture in
hyperbolic geometry, in Seminar on Minimal Submanifolds, Editor E.~Bombieri, {\it Ann. of Math. Stud.},  no.~103, 1983, 339--358.

\bibitem{W1} Wei S.W.,   Minimality, stability, and Plateau's problem, Thesis, Berkeley,
1980.

\bibitem{W} Wei S.W.,  Plateau's problem in symmetric
spaces, {\it Nonlinear Anal.} {\bf 12} (1988),  749--760.

\bibitem{Y} Yau S.T.,  Some function theoretic properties of complete
Riemannian manifolds and their applications to geometry, {\it Indiana
Univ. Math. J.} {\bf 25} (1976), 659--670.
\end{thebibliography}
\end{document}